\documentclass[a4paper,10pt,notitlepage,reqno]{article}
\usepackage[english]{babel}
\usepackage[latin1]{inputenc}
\usepackage[T1]{fontenc}
\usepackage{amssymb,amsthm,amsmath} 
\usepackage{mathtools}
\usepackage{mathrsfs}
\usepackage{enumerate}
\usepackage{authblk}
\usepackage{stackrel}
\usepackage{eucal}
\usepackage{color}
\usepackage{geometry}
\usepackage{hyperref}
\numberwithin{equation}{section}
\newcommand{\ie}{\emph{i.e.}}

\newcommand{\cf}{\emph{cf.}}
\newcommand{\etal}{\emph{et al.}}
\newcommand{\Com}{\mathbb{C}}
\newcommand{\Real}{\mathbb{R}}

\newcommand{\sgn}{\mathop{\mathrm{sgn}}\nolimits}
\newcommand{\supp}{\mathop{\mathrm{supp}}\nolimits}

\newcommand{\esssup}{\mathop{\mathrm{ess\;\!sup}}}
\newcommand{\Dom}{\mathsf{D}}
\newcommand{\eps}{\varepsilon}
\newcommand{\sii}{L^2}
\newtheorem{Theorem}{Theorem}
\newtheorem{Corollary}{Corollary}
\newtheorem{Lemma}{Lemma}
\theoremstyle{definition}

\newtheorem{Remark}{Remark}

\usepackage{color}

\definecolor{DarkBlue}{rgb}{0,0.1,0.7}
\definecolor{DarkGreen}{rgb}{0,0.5,0.1}

\newcommand\soutD{\bgroup\markoverwith
{\textcolor{DarkBlue}{\rule[.01ex]{2pt}{1pt}}}\ULon}
\newcommand{\Hm}[1]{\leavevmode{\marginpar{\tiny%
$\hbox to 0mm{\hspace*{-0.5mm}$\leftarrow$\hss}%
\vcenter{\vrule depth 0.1mm height 0.1mm width \the\marginparwidth}%
\hbox to
0mm{\hss$\rightarrow$\hspace*{-0.5mm}}$\\\relax\raggedright #1}}}

\begin{document}
%
\title{\textbf{\Large 
Sharp bounds for eigenvalues of biharmonic 
operators with complex potentials in low dimensions}}

\author{Orif~O.~Ibrogimov,$^a$
	\ David Krej\v{c}i\v{r}{\'\i}k\,$^a$
	\ and \ Ari Laptev\,$^b$}
\date{\small 
\begin{quote}
\emph{
	\begin{itemize}
		\item[$a)$] 
		Department of Mathematics, Faculty 
		of Nuclear Sciences and Physical 
		Engineering, Czech Technical University 
		in Prague,
		Trojanova 13, 12000 Prague 2, Czech 
		Republic; \\ibrogori@fjfi.cvut.cz, 
		david.krejcirik@fjfi.cvut.cz.%
		\item[$b)$] 
		Department of Mathematics, Imperial 
		College London, Huxley Building, 180 
		Queen's Gate, London SW7 2AZ, UK; 
		a.laptev@imperial.ac.uk.%
		\end{itemize}
	}
\end{quote}
5 March 2019}
\maketitle
\begin{abstract}
We derive sharp quantitative bounds for 
eigenvalues of biharmonic operators 
perturbed by complex-valued potentials 
in dimensions one, two and three. 
\end{abstract}

\section{Introduction}
%
Spectral theory of self-adjoint operators  
has exhibited an enormous development 
since the discovery of quantum mechanics 
in the beginning of the last century
and by now it can be regarded as well understood in many respects.
Recent years have brought new motivations
for considering \emph{non-self-adjoint} operators, too,
including quantum mechanics again, but this theory is still in its infancy. 

A strong impetus for a systematic study of spectral properties 
of Schr\"odinger operators with complex-valued potentials 
goes back to the celebrated result of Davies \etal\ 
in 2001~\cite{Abramov-Aslanyan-Davies_2001}
showing that every discrete eigenvalue of the 
one-dimensional operator $-\Delta\dot{+}V$ lies in the closed disk
of the complex plane centred at the origin 
and with the radius equal 
to~$\frac{1}{4}\|V\|^2_{L^1(\Real)}$. 
This bound is sharp in the sense that 
there are potentials with eigenvalues
lying on the boundary circle. 
The purpose of this paper is to extend this type of bounds
to biharmonic operators  
\begin{equation}\label{bi-intro}
  H_V := \Delta^2\dot{+}V
  \qquad \mbox{in} \qquad
  \sii(\Real^d)
\end{equation}
in dimensions $d=1, 2, 3$.
The feature of our results is that our bounds are sharp 
and quantitative (at least if $d=1,3$).
Moreover, potentials of minimal regularity are considered
and we cover embedded eigenvalues, too. 
 
The pioneering work~\cite{Abramov-Aslanyan-Davies_2001}
has been followed by extensive investigations 
devoted to bounds on the number and  
magnitudes of eigenvalues of Schr\"odinger operators
with complex-valued potentials.
Instead of listing the huge number of articles by various groups
on operators not dealt with in this paper,
we refer to the recent work~\cite{FKV} 
containing a rather extensive bibliography. 
On the other hand, related issues for
higher-order partial differential operators 
seem to be just scarcely considered in the literature,
even in the self-adjoint setting. 
The reader is referred to 
\cite{Netrusov-Weidl-1996, Weidl-CPDE-1999,
Ekholm-Enblom-2010, Kiik-Kurasov-Usman2015, Enblom-LMP-2016, Cuenin-JFA-2017, Hulko-RMP-2018}; 
see also~\cite{Evans-Lewis-2005, 
Arazy-Zelenko-2006, Arazy-Zelenko-II-2006,  
Foerster-Ostensson-MNr-2008, Arrieta.et.al-IEOT-2017}
for other spectral questions.

We recall the importance of biharmonic operators 
in continuum mechanics and in particular linear elasticity theory 
and the solution of Stokes flows.
To the best of our knowledge, 
the only work where eigenvalue bounds for non-self-adjoint 
biharmonic operators were considered previously 
is the recent paper of Enblom~\cite{Enblom-LMP-2016}. 
We also refer to recent Hulko's paper~\cite{Hulko-RMP-2018} 
for bounds on the number of discrete eigenvalues 
of biharmonic operators perturbed by exponentially decaying 
complex-valued potentials.
While the author of~\cite{Enblom-LMP-2016} proceeds in 
a rather extensive generality (polyharmonic operators in any dimension),
her bounds are not quantitative 
(in the sense that the involved constants depend on the arguments 
of the individual eigenvalues under consideration)
and embedded eigenvalues are not covered. 
For these reasons, it is our belief that our results  
are of interest even if restricted to biharmonic operators
in low dimensions only.
Moreover, our methodology in principle enables one to obtain analogous 
results also for higher-order polyharmonic operators. 
 
Before stating our main results, 
we remark that the operator~\eqref{bi-intro} 
is introduced in a standard way 
as an m-accretive operator obtained as a form sum of
the bi-Laplacian $H_0 := \Delta^2$ with domain $H^2(\Real^d)$
and a relatively form-bounded potential $V:\Real^d \to \Com$.
We refer to Section~\ref{Sec.pre} for more details.
More specifically, our standing assumption is that 
there exist numbers $a\in(0,1)$ and $b\in\Real$ such that, 
for all $\psi\in H^2(\Real^d)$,
\begin{equation}\label{rel.bddness}
	\int_{\Real^d} |V||\psi|^2 \leq 
	a\int_{\Real^d} |\Delta\psi|^2 
		 + b\int_{\Real^d} |\psi|^2.
\end{equation}
In particular, potentials $V \in L^1(\Real^d)$ are covered by this hypothesis.
Furthermore, in the same way, 
it is also possible to proceed in a greater generality
and give a meaning to the distributional Dirac delta potential~$\delta$,
which is explicitly solvable.
Putting $\|\delta\|_{L^1(\Real^d)}:=1$ by convention,
Theorem~\ref{Thm.EV.bound.unif} remains valid in this 
more general setting.

Our first result is a biharmonic analogue of the celebrated result 
of~\cite{Abramov-Aslanyan-Davies_2001}
for one-dimensional Schr\"odinger operators.       
\begin{Theorem}\label{Thm.EV.bound.unif}
Let $d\in\{1,2,3\}$ and assume that $V \in L^1(\Real^d)$. 
Then there exists a universal constant $C_d>0$ such that 
\begin{equation}\label{opt.bound.unif}
	\displaystyle
	\sigma_\mathrm{p}(H_V) 
	\subset
	\left\{
	\lambda \in \Com : \
	|\lambda| \leq C_d \, 
        \|V\|_{L^1(\Real^d)}^{4/(4-d)}
	\right\}.
\end{equation}
Moreover, one can take $C_1=\frac{1}{4}$ 
and~$C_3=\frac{1}{4}\frac{1}{(4\pi)^4}$ for
dimensions~$d=1$ and~$d=3$, respectively.
\end{Theorem}

The theorem is sharp due to the weak-coupling asymptotics
$$
  \inf\sigma(H_{\beta V}) 
  = -c_d \, \beta^{4/(4-d)} \, \|V\|_{L^1(\Real^d)}^{4/(4-d)} 
  + o(\beta^{4/(4-d)})
  \qquad \mbox{as} \qquad
  \beta \to 0^+
$$
valid for every sufficiently regular,
non-positive, real-valued potential~$V$
with $c_1:=\frac{1}{4}$, $c_2:=\frac{1}{64}$ 
and $c_3:=\frac{1}{4}\frac{1}{(4\pi)^4}$,  
\cf~\cite[Eq.~(8)]{Netrusov-Weidl-1996}.
We do not know whether one can take $C_2 = \frac{1}{64}$
in \eqref{opt.bound.unif} for dimension $d=2$,
but it is natural to conjecture so.
The result is sharp also in the sense that 
for $d=1$ (respectively, $d=3$) one has
$
  \sigma_\mathrm{p}(H_{\alpha\delta})
  = \{\frac{1}{4}\alpha^{4/3}e^{-i\pi/3}\} 
$
for every $\alpha \in \Com$ such that $\Re\alpha < |\Im\alpha|$
(respectively, 
$
  \sigma_\mathrm{p}(H_{4\pi\alpha\delta})
  \supset \{-\frac{1}{4}\alpha^{4}\} 
$
if $\Re\alpha < -|\Im\alpha|$).
	
In the next series of results, we go beyond the $L^1$-potentials
in the three-dimensional case.
First, we employ the Rollnik class of potentials 
(\cf~\cite[Ch.~I]{SiQF}),
which consist of all $V\in L^1_\mathrm{loc}(\Real^3)$ such that
\begin{equation}\label{Rollnik.norm}
	\|V\|_{R(\Real^3)}^2 := 
	\iint_{\Real^3\times\Real^3} 
	\frac{|V(x)||V(y)|}{|x-y|^2}\,dx\,dy<\infty
	.
\end{equation}
In Section~\ref{Sec.pre}, we argue that $V\in R(\Real^3)$
is a sufficient condition to guarantee~\eqref{rel.bddness},
so~$H_V$ is well defined.

\begin{Theorem}\label{Thm.EV.bound.3D.Rollnik}
Let $d=3$ and assume that $V \in R(\Real^3)$. 
Then
\begin{equation}\label{inclusion.Rollnik}
	\displaystyle
	\sigma_\mathrm{p}(H_V)  
	\subset
	\left\{\lambda \in \Com : \
	|\lambda| \leq \, \frac{1}{2}\frac{\|V\|^2_{R(\Real^3)}}{(4\pi)^2}
	\right\}
	\,.
\end{equation}
\end{Theorem}

This result has an important consequence
for potentials belonging to 
$L^{3/2}(\Real^3) \hookrightarrow R(\Real^3)$.
In fact, using a sharp version of 
the Hardy-Littlewood-Sobolev inequality (\cf~\cite[Thm.~4.3]{LL})
which quantifies the embedding, 
the bound~\eqref{inclusion.Rollnik} immediately implies the following result.

\begin{Corollary}\label{Cor.opt.EV.bound.3D}
Let $d=3$ and assume that $V\in L^{3/2}(\Real^3)$. 
Then
\begin{equation}\label{opt.bound.3D}
	\displaystyle
	\sigma_\mathrm{p}(H_V)  
	\subset
	\left\{
	\lambda \in \Com : \
	|\lambda| \leq \frac{1}{8}
	\frac{1}{(4\pi)^{2/3}}
	\|V\|^2_{L^{3/2}(\Real^3)}
	\right\}.
\end{equation}
\end{Corollary}

Finally, we establish the following robust result. 

\begin{Theorem}\label{Thm.alt.EV.bound.3D}
Let $d=3$ and assume \eqref{rel.bddness}. 
Then
\begin{equation}\label{bound.3D}
	\displaystyle
	\sigma_\mathrm{p}(H_V)  
	\subset
	\left\{
	\lambda \in \Com : \
	|\lambda| \leq 
	\left(
	\sup_{\stackrel[\psi\not=0]{}{\psi \in H^1(\Real^3)}}	 
	\frac{\displaystyle \int_{\Real^3} |V||\psi|^2}
	{\displaystyle \int_{\Real^3} |\nabla\psi|^2} 
	\right)^2
	\right\}\,.
\end{equation}
\end{Theorem}

This theorem is particularly useful to cover 
the Hardy-type potentials $x \mapsto |x|^{-2}$,
which belong neither to $L^1(\Real^3)$ nor $R(\Real^3)$.
For this reason, let us consider the class of potentials,
which consist of all $V\in L^1_\mathrm{loc}(\Real^3)$ such that 
\begin{equation}\label{Hardy.norm}
\|V\|_{H(\Real^3)}:=\esssup_{x\in\Real^3}	 
|x|^2 |V(x)| <\infty.	
\end{equation}
Employing the classical Hardy inequality
\begin{equation}\label{Ineq.Hardy}
\int_{\Real^3}|\nabla\psi|^2 
\geq \frac{1}{4}
\int_{\Real^3}\frac{|\psi(x)|^2}{|x|^2} \, dx
\end{equation}
valid for every $\psi \in H^1(\Real^3)$,
it is easy to check that $V \in H(\Real^3)$
is a sufficient condition to guarantee~\eqref{rel.bddness}. 
Furthermore, using~\eqref{Ineq.Hardy} in~\eqref{bound.3D},
we get the following immediate consequence.
\begin{Corollary}
Let $d=3$ and assume that $V \in H(\Real^3)$. 
Then
\begin{equation}\label{Cor.3D.Hardy.pot}
	\displaystyle
	\sigma_\mathrm{p}(H_V)  
	\subset
	\left\{
	\lambda \in \Com : \
	|\lambda| \leq 16 \|V\|^2_{H(\Real^3)}
	\right\}\,.
\end{equation}
\end{Corollary} 

In this paper we leave aside dimension $d=4$,
which exhibits the same type of difficulties as dimension two
for Schr\"odinger operators.
The reason for explicit constants in our one- and three-dimensional
theorems is due to the availability of explicit forms of 
the resolvent kernels of~$H_0$ in these dimensions.
On the other hand, in view of the subcriticality of~$H_0$ 
in dimensions $d \geq 5$, a different kind of results is expected
in analogy with the case of Schr\"odinger operators 
(\cf~\cite{Kato_1966,Frank_2011,FKV}):
the point spectrum of~$H_V$ should be empty for all potentials 
which are sufficiently small in a suitable sense.
This type of results are also left open in this paper.

The organisation of the paper is as follows.
In Section~\ref{Sec.pre} we rigorously introduce
the biharmonic operators~\eqref{bi-intro}
and determine sufficient conditions on the potential~$V$
which guarantee the basic hypothesis~\eqref{rel.bddness}.  
In Section~\ref{Sec.BS} we establish the main tool of our approach,
namely a sort of the Birman-Schwinger principle covering
embedded eigenvalues, too. 
The proofs of Theorems~\ref{Thm.EV.bound.unif}--\ref{Thm.alt.EV.bound.3D} 
are given in Section~\ref{Sec.proofs}.
Finally, in Section~\ref{Sec.end} we introduce explicitly solvable 
models in terms of Dirac delta potentials 
and argue about the optimality of 
Theorem~\ref{Thm.EV.bound.unif} for dimensions~$d=1$ and $d=3$.

\section{Definition of biharmonic operators}\label{Sec.pre}
%
First of all, let us comment on the definition of~$H_V$ 
given in~\eqref{bi-intro}.

Let~$H_0$ be the the self-adjoint operator in $\sii(\Real^d)$ 
associated with the quadratic form 
\begin{equation}
	h_0[\psi]:=\int_{\Real^d} |\Delta\psi|^2, 
	\qquad \Dom(h_0):=H^2(\Real^d).
\end{equation}
One has $\Dom(H_0) = H^4(\Real^d)$ and $H_0=\Delta^2$. 
The spectrum of~$H_0$ is purely absolutely continuous
and coincides with the semi-axis $[0,+\infty)$.
If $d \leq 4$, the operator~$H_0$ is \emph{critical}
in the sense that $\inf\sigma(H_0+V) < 0$ 
whenever $V \in C_0^\infty(\Real^d)$ is real-valued,
non-positive and non-trivial.
On the other hand, if $d \geq 5$,
the operator is \emph{subcritical} in the sense
that its spectrum is stable under the non-positive 
perturbations above;
more specifically, $H_0$~satisfies a Hardy-type inequality
$H_0 \geq \rho$ in the sense of forms, 
where $\rho:\Real^d \to [0,\infty)$ is non-trivial.

Let~$v$ be a quadratic form in $\sii(\Real^d)$, 
which is relatively bounded with respect to~$h_0$ 
with the relative bound less than one.
That is, $\Dom(v) \supset H^2(\Real^d)$
and there exist numbers $a\in(0,1)$ and $b\in\Real$ such that, 
for all $\psi\in H^2(\Real^d)$,
\begin{equation}\label{rb}
	|v[\psi]| \leq 
	a\int_{\Real^d} |\Delta\psi|^2 
		 + b\int_{\Real^d} |\psi|^2.
\end{equation}
Then the sum $h_V := h_0+v$ is a 
closed sectorial form with $\Dom(h_V)=H^2(\Real^d)$,
which gives rise to an m-sectorial operator~$H_V$ in $L^2(\Real^d)$ 
via the representation theorem 
(\cf~\cite[Thm.~VI.2.1]{Kato}).  

For example, if $V \in L^1_\mathrm{loc}(\Real^d)$ is such that
$$
  v[\psi]:=\int_{\Real^d} V |\psi|^2
  \,, \qquad
  \Dom(v):=
  \left\{
  \psi\in L^2(\Real^d): 
	\int_{\Real^d}|V||\psi|^2<\infty
  \right\}
  ,
$$
verifies~\eqref{rb} (which coincides with~\eqref{rel.bddness} in this case), 
then we write $H_V = H_0 \dot{+} V$ as in~\eqref{bi-intro}
and understand the plus symbol in the sense of forms described above.
It is important to keep in mind that~$H_V$ may differ 
from the operator sum $H_0+V$. 
Since the adjoint operator satisfies 
$H^*_V=H_{\overline{V}}=\mathcal{T}H_V\mathcal{T}$, 
where~$\mathcal{T}$ is the complex conjugation 
operator defined by $\mathcal{T}\psi:=\overline{\psi}$, 
$H_V$ is a $\mathcal{T}$-self-adjoint operator. 
Consequently, the residual spectrum of~$H_V$ is always empty 
(\cf~\cite[Sec.~III.5]{Edmunds-Evans}). 
Furthermore, if~$V$ vanishes at infinity in a suitable sense, 
then (any kind of) the essential spectrum of~$H_V$ 
coincides with the semi-axis $[0,+\infty)$. 

Let us now discuss sufficient conditions which guarantee~\eqref{rb}.

By the Sobolev embedding theorem (\cite[Thm.~5.4]{Adams}),
every function $\psi \in H^2(\Real^d)$ is bounded and continuous
as long as $d \leq 3$.
More specifically (\cf~\cite[Theorem~IX.28]{RS2}),
for any positive~$\alpha$ there is $\beta \in \Real$ 
such that, for all $\psi \in H^2(\Real^d)$,
\begin{equation}\label{elliptic}
	\|\psi\|_{L^{\infty}(\Real^d)}
	\leq \alpha\|\Delta\psi\|_{\sii(\Real^d)}
        +\beta\|\psi\|_{\sii(\Real^d)}
  \,.
\end{equation}
Consequently, any potential $V \in L^1(\Real^d) + L^\infty(\Real^d)$
satisfies~\eqref{rel.bddness} with the relative bound equal to zero
(\ie~$a$ can be chosen arbitrarily small).  
Furthermore, the inequality~\eqref{elliptic} enables us 
to give a meaning to the operator $H_\delta = H_0 \dot{+} \delta$,
where~$\delta$ is the Dirac delta function,
by setting $v[\psi] := |\psi(0)|^2$, $\Dom(v) := H^2(\Real^d)$.
  
Rollnik potentials are relatively 
form-bounded with respect to the Laplacian
with the relative bound equal to zero
(\cf~\cite[Theorem~X.19]{RS2}). 
More specifically, if $V \in R(\Real^3)$,
then for any positive~$\alpha$ there is $\beta \in \Real$ 
such that, for all $\psi \in H^1(\Real^3)$,
\begin{equation*} 
	\int_{\Real^3} |V| |\psi|^2
	\leq \alpha\|\nabla\psi\|_{\sii(\Real^3)}^2
        +\beta\|\psi\|_{\sii(\Real^3)}^2
  \,.
\end{equation*}
Consequently, if additionally $\psi \in H^2(\Real^3)$,
then 
\begin{equation}\label{pp} 
\begin{aligned}
	\int_{\Real^3} |V| |\psi|^2
	&\leq \alpha \, (\psi,-\Delta\psi)_{\sii(\Real^3)}
        +\beta\|\psi\|_{\sii(\Real^3)}
  \\
  &\leq \alpha \, \|\psi\|_{\sii(\Real^3)} \|\Delta\psi\|_{\sii(\Real^3)}
        +\beta\|\psi\|_{\sii(\Real^3)}^2
  \,.
\end{aligned}
\end{equation}
From this inequality, it is easy to conclude that
any potential $V \in R(\Real^3) + L^\infty(\Real^3)$
satisfies~\eqref{rel.bddness} with the relative bound equal to zero.

Finally, if $V \in H(\Real^3)$, 
then the Hardy inequality~\eqref{Ineq.Hardy} yields,
for all $\psi \in H^1(\Real^3)$,
\begin{equation*} 
\begin{aligned}
	\int_{\Real^3} |V| |\psi|^2
  \leq \|V\|_{H(\Real^3)} \int_{\Real^3} \frac{|\psi(x)|^2}{|x|^2} \, dx
  \leq 4 \|V\|_{H(\Real^3)} \|\nabla\psi\|_{\sii(\Real^3)}^2
  \,.
\end{aligned}
\end{equation*}
Hence $V$ is form-subordinated with respect to the Laplacian.
Proceeding as in~\eqref{pp}, we conclude that 
any potential $V \in H(\Real^3) + L^\infty(\Real^3)$
satisfies~\eqref{rel.bddness} with the relative bound equal to zero.

%
\section{The Birman-Schwinger principle}\label{Sec.BS}
%
The main role in our proof of 
Theorems~\ref{Thm.EV.bound.unif}--\ref{Thm.alt.EV.bound.3D} 
is played by the Birman-Schwinger operator
\begin{equation}\label{BS.op}
K_z := |V|^{1/2} \, (H_0-z)^{-1} \, V_{1/2}
\qquad \mbox{with} \qquad
V_{1/2} := |V|^{1/2} \, \sgn(V) 
\,,
\end{equation}
where $\sgn\colon\Com\to\Com$ is the complex 
signum function defined by 
\begin{equation*}
	\sgn(z):=
	\begin{cases}
	\displaystyle\frac{z}{|z|} \quad 
	&\mbox{ if } z \not= 0,
	\\[1ex]
	\;\,0 \quad &\mbox{ if } z=0. 
	\end{cases}
\end{equation*}
We abuse the notation by using the same 
symbols for maximal operators of multiplication
and their generating functions.
The operator~$K_z$ is well defined 
(on its natural domain of the composition of 
three operators) for all $z\in\Com$ and $d\geq1$.

If $z \not \in [0,+\infty)$, however,
$K_z$~is a bounded operator under our hypothesis~\eqref{rel.bddness}. 
Indeed, $V_{1/2}$ maps $\sii(\Real^d)$ to $H^{-2}(\Real^d)$ by duality,
$(H_0-z)^{-1}$ is an isomorphism between $H^{-2}(\Real^d)$ and $H^2(\Real^d)$,
and the latter space is mapped by~$|V|^{1/2}$ 
back to $\sii(\Real^d)$.     
Furthermore, we have a useful formula for the integral 
kernel of~$K_z$:
\begin{equation}\label{K.op}
	K_z(x,y) = |V|^{1/2}(x) 
	\, \widetilde G_z(x,y) 
	\, V_{1/2}(y)\,,
\end{equation}
where~$\widetilde G_z$ is the 
\emph{Green's function} of~$H_0-z$,
\ie~the integral kernel of the 
resolvent~$(H_0-z)^{-1}$. 
What is more, for dimensions~$d=1$ and~$d=3$, 
we have explicit formulae for~$\widetilde G_z$, whereas 
the latter can be expressed in terms of modified Hankel 
functions for $d=2$. 
In fact, using the identity
\begin{equation}\label{idea}
(\Delta^2-z)^{-1} 
=\frac{1}{2k}
\left[(-\Delta-k)^{-1} 
-(-\Delta+k)^{-1}\right] ,
\end{equation}
where $k\in\Com\setminus[0,\infty)$ 
is such that $k^2=z$ 
(throughout the paper we choose the principal branch of the square root),
and the well-known formulae for the integral kernels of the 
resolvent of the Laplacian in these dimensions yield
\begin{equation}\label{Green.BiHm}
	\widetilde G_z(x,y) = \frac{1}{2k}\bigl[G_k(x,y)-G_{-k}(x,y)
				\bigr],
\end{equation}
where 
\begin{equation}\label{Green.Lap}
	G_{k}(x,y) := 
	\begin{cases}
	\displaystyle
	\frac{e^{-\sqrt{-k}\,|x-y|}}{2\sqrt{-k}} 
	\quad & \mbox{ if } \quad d=1,\\[2ex]
	\displaystyle\frac{1}{2\pi} \, K_0\bigl(\sqrt{-k}\,|x-y|\bigr) 
	\quad & \mbox{ if } \quad d=2,\\[2ex]
	\displaystyle\frac{1}{4\pi} \frac{e^{-\sqrt{-k}\,|x-y|}}{|x-y|} 
	\quad & \mbox{ if } \quad d=3. 
	\end{cases}
\end{equation}
Here~$K_0$ is a modified Bessel (also called Macdonald's)
function (see~\cite[Sec.~9]{Abramowitz-Stegun}
or \cite[Sec.~8.4--8.5]{Gradshteyn-Ryzhik-2007}).

The following lemma provides an integral 
criterion for the existence of solutions to 
the differential eigenvalue equation 
corresponding to~$H_V$ and can be considered 
as a one-sided version of the conventional Birman-Schwinger 
principle extended to possible eigenvalues 
embedded in $[0,+\infty)$ as well. Its proof is 
inspired by the ones of the similar results
in~\cite{FKV,FK-Dirac-LMP2019}. 
By $\varphi\in L^2_0(\Real^d)$ in the 
lemma below we mean $\varphi\in L^2(\Real^d)$ 
and that $\supp\varphi$ is compact.
\begin{Lemma}\label{Lem.BS}
Let $d \in \{1,2,3\}$ and assume~\eqref{rel.bddness}. 
If $H_V\psi=\lambda\psi$ with some 
$\lambda\in\Com$ and $\psi\in\Dom(H_V)$, 
then $\phi := |V|^{1/2}\psi \in L^2(\Real^d)$ 
obeys
\begin{equation}\label{BS}
	\forall \varphi \in L^2_0(\Real^d)
	\,, \qquad
	\lim_{\eps \to 0^\pm} 
		(\varphi,K_{\lambda+i\eps} \phi) 
		= - (\varphi,\phi) 
	\,.
\end{equation}
\end{Lemma}
\begin{proof}
First, we notice that \eqref{rel.bddness} implies that 
	$\phi\in L^2(\Real^d)$ whenever 
	$\psi\in\Dom(H_V)\subset H^2(\Real^d)$. 
	Let $\varphi \in L^2_0(\Real^d)$ be fixed.		
Given any $\lambda \in \Com$, there is 
$\eps_0>0$ such that 
$\lambda + i\eps \not \in [0,+\infty)$ 
for all real~$\eps$ satisfying $0<|\eps|<\eps_0$.  
We have 
\begin{equation}\label{i1}
	\begin{aligned}
	(\varphi,K_{\lambda + i\eps} \phi)
	&= \iint_{\Real^d\times\Real^d} 
	\overline{\varphi(x)} \, |V|^{1/2}(x) \, 
	G_{{\lambda + i\eps}}(x,y) \, V(y) \, \psi(y) 
	\, dx \, dy\\
	&= \int_{\Real^d} \eta_\eps(y) \, V(y) \, \psi(y) \, dy
	\,,
	\end{aligned}
\end{equation}
where
\begin{equation*}
	\eta_\eps := \int_{\Real^d} 
	\overline{\varphi(x)} \, |V|^{1/2}(x) \, 
	G_{{\lambda + i\eps}}(x,\cdot) \, dx
	= (H_0-\lambda-i\eps)^{-1} \, |V|^{1/2} \, \overline{\varphi}
	\,,
\end{equation*}
where the second equality holds due to 
the symmetry $G_z(x,y)=G_z(y,x)$.
In view of~\eqref{rel.bddness}, $|V|^{1/2} \overline{\varphi} \in \sii(\Real^d)$.
Since $\eps \not=0$ is so small that 
$\lambda + i\eps \not\in \sigma(H_0)$,
we have 
$\eta_\eps \in \Dom(H_0) = H^4(\Real^d)$.
In particular, $\eta_\eps \in H^2(\Real^d)$ 
and the weak formulation of the eigenvalue 
equation $H_V\psi=\lambda\psi$ yields
\begin{equation}\label{i2}
	\begin{aligned}
	\int_{\Real^d} \eta_\eps(y) \, V(y) \, \psi(y) \, dy
	&= -(\Delta\overline{\eta_\eps},\Delta\psi) 
	+ \lambda \, (\overline{\eta_\eps},\psi)
	\\
	&= -(\Delta\overline{\psi},\Delta\eta_\eps) 
	+ \lambda \, (\overline{\psi},\eta_\eps)
	\\
	&= -(\Delta\overline{\psi},\Delta\eta_\eps) 
	+ (\lambda+i\eps) \, (\overline{\psi},\eta_\eps)
	- i\eps \, (\overline{\psi},\eta_\eps)
	\\
	&= -(\overline{\psi},|V|^{1/2}\overline{\varphi}) 
	- i\eps \, (\overline{\psi},\eta_\eps)
	\\
	&= -(\varphi,|V|^{1/2}\psi) 
	- i\eps \, (\overline{\eta_\eps},\psi)
	\,.
	\end{aligned}
\end{equation}
Here the last but one equality follows from 
the weak formulation of the resolvent equation
$
(H_0-\lambda-i\eps)\eta_\eps = |V|^{1/2} \overline{\varphi}
$.
Consequently, \eqref{i1} and~\eqref{i2} 
imply~\eqref{BS} after taking the limit 
$\eps \to 0^\pm$, provided that 
$\eps \, (\bar\eta_\varepsilon,\psi) \to 0$
as $\eps \to 0$. To see the latter, we write
\begin{equation*}
	|(\overline{\eta_\eps},\psi)| 
	= |(\varphi,M_\eps\psi)|
	\leq \|\varphi\| \;\! \|M_\eps\| \;\! \|\psi\|
	\,,
\end{equation*}
where
$
M_\eps := \chi_\Omega \, |V|^{1/2} (H_0-\lambda-i\eps)^{-1}
$
with $\Omega := \supp\varphi$,
and it remains to show that $\eps \, \|M_\eps\|$ 
tends to zero as $\eps \to 0$.
Following~\cite[Thm.~III.6]{SiQF}, we use the 
resolvent kernels~\eqref{Green.BiHm} with~\eqref{Green.Lap} 
and the estimate 
$\|M_\eps\| \leq \|M_\eps\|_\mathrm{HS}$.

For $d=1$, we have
\begin{equation*}
\begin{aligned}
	\|M_\eps\|_\mathrm{HS}^2
	&= \frac{1}{4|k|^2} \iint_{\Omega \times \Real} 
	|V(x)| \, |G_k(x,y)-G_{-k}(x,y)|^2
	\, dx \, dy\\
	&\leq\frac{1}{2|k|^2}\iint_{\Omega \times \Real} 
	|V(x)| \, \bigg[\frac{e^{-2 \;\!\Re\sqrt{-k} \;\! |x-y|}}{4|k|}
	+
	\frac{e^{-2 \;\! \Re\sqrt{k} \;\! |x-y|}}{4|k|}\bigg]
	\, dx \, dy\\
	&=\frac{1}{8|k|^3}\bigg[\frac{1}{\Re\sqrt{-k}}+\frac{1}{\Re\sqrt{k}}\bigg]
	\int_\Omega |V(x)| \, dx
	\,,
\end{aligned}
\end{equation*}
where the last integral is bounded because 
$V \in L_\mathrm{loc}^1(\Real)$
as a consequence of~\eqref{rel.bddness}. 
Elementary calculations show that
$|k|^3\Re\sqrt{\pm k}$ can not decay faster 
than $|\eps|^{7/4}$ as $\eps\to0$. Hence, 
we have $\|M_\eps\|=\mathcal{O}(|\eps|^{-7/8})$ 
as $\eps\to0$, which concludes the proof of the 
lemma for $d=1$.

For $d=3$, we have analogous estimates
\begin{equation*}
\begin{aligned}
	\|M_\eps\|_\mathrm{HS}^2
	&\leq\frac{1}{32\pi^2|k|^2}\iint_{\Omega \times \Real^3} 
	|V(x)| \, \bigg[\frac{e^{-2 \;\!\Re\sqrt{-k} \;\! |x-y|}}{|x-y|^2}
	+
	\frac{e^{-2 \;\! \Re\sqrt{k} \;\! |x-y|}}{|x-y|^2}\bigg]
	\, dx \, dy\\
	&=\frac{1}{16\pi|k|^2}
	\bigg[\frac{1}{\Re\sqrt{-k}}+\frac{1}{\Re\sqrt{k}}
	\bigg]\int_\Omega |V(x)| \, dx
	\,.
\end{aligned}
\end{equation*}
Since $|k|^2\Re\sqrt{\pm k}$ can not decay 
faster than $|\eps|^{5/4}$ as $\eps\to0$, 
we have $\|M_\eps\|=\mathcal{O}(|\eps|^{-5/8})$ 
as $\eps\to0$. This concludes the proof of the 
lemma for $d=3$.

For $d=2$, we need to do rather involved 
estimates. Using the integral representation of 
the Macdonald's function 
(\cf~\cite[\S~8.432]{Gradshteyn-Ryzhik-2007}), 
we can write Green's function for the biharmonic operator as 
\begin{equation*}
	\displaystyle
	\widetilde G_z(x,y)=\frac{1}{4\pi k}\int_1^\infty 
		\frac{e^{-\sqrt{-k}\,|x-y|t} - e^{-\sqrt{k}\,|x-y|t}}{\sqrt{t^2-1}}\, dt.
\end{equation*}
Next, we observe that, for $t>1$,
\begin{equation*}
\begin{aligned}
	\int_{0}^{\infty}\bigl|e^{-\sqrt{-k}ts}
	+e^{-\sqrt{k}ts}\bigr|^2\, ds 
	\leq 2 \int_{0}^{\infty}\bigl(e^{-2\Re\sqrt{-k}ts}
	+e^{-2\Re\sqrt{k}ts}\bigr)\, ds
	=\frac{1}{t}\bigg[\frac{1}{\Re\sqrt{k}}+\frac{1}{\Re\sqrt{-k}}
	\bigg].
\end{aligned}
\end{equation*}
In view of this estimate, the Minkowski inequality yields 
\begin{equation*}
\begin{aligned}
	\|M_\eps\|_\mathrm{HS}^2 
	&\leq\frac{1}{8\pi^2|k|^2}\bigg(\int_{\Omega} |V(x)|\, dx\bigg)
	\bigg(\int_1^{\infty}\bigg(\int_0^{\infty}\bigl|e^{-\sqrt{-k}ts}
	+e^{-\sqrt{k}ts}\bigr|^2 ds\bigg)^{1/2}\frac{dt}{\sqrt{t^2-1}}\bigg)^2\\
	&=\frac{1}{8\pi^2|k|^2}
	\bigg[\frac{1}{\Re\sqrt{-k}}+\frac{1}{\Re\sqrt{k}}
	\bigg]\bigg(\int_\Omega |V(x)| \, dx\bigg) 
	\bigg(\int_1^{\infty}\frac{1}{\sqrt{t}\sqrt{t^2-1}}\,dt\bigg)^2
	\,.
\end{aligned}
\end{equation*}
We already know that $|k|^2\Re\sqrt{\pm k}$ can not decay 
faster than $|\eps|^{5/4}$ as $\eps\to0$. Hence, again we have $\|M_\eps\|=\mathcal{O}(|\eps|^{-5/8})$ 
as $\eps\to0$, which concludes the proof of the 
lemma for $d=2$.
\end{proof}

The preceding lemma can be viewed as a precise 
statement of one side of the Birman-Schwinger 
principle under the minimal regularity 
assumption~\eqref{rel.bddness} on the potential. 
It roughly says that if~$\lambda$ is an eigenvalue 
of~$H_V$, then $-1$~is an eigenvalue of an 
integral equation related to~$K_\lambda$. 
If $\lambda \not\in\sigma(H_0)$ the converse 
implication also holds, but it is not generally 
true if $\lambda \in\sigma(H_0)$ 
(\cf~\cite[Sec.~III.2]{SiQF}) 
and it is not needed for the purpose of this 
paper. In fact, we exclusively use the 
following corollary of Lemma~\ref{Lem.BS}.

\begin{Corollary}\label{Corol.BS}
Let $d \in \{1,2,3\}$ and assume~\eqref{rel.bddness}. 
Let $\lambda\in\Com$ be arbitrary. If either 
$\lambda \in \rho(H_0)$ and $\|K_{\lambda}\|<1$, 
or $\lambda \in \sigma(H_0)$ and 
$
\displaystyle
\liminf_{\eps\to 0^\pm} \|K_{\lambda+i\eps}\|<1
$, 
then $\lambda\notin\sigma_\mathrm{p}(H_V)$. 
\end{Corollary}
\begin{proof}
Let $\lambda \in \sigma_\mathrm{p}(H_V)$, 
let~$\psi$ be a corresponding eigenfunction and set 
$\phi: = |V|^{1/2}\psi$. If it were the case 
	that $\phi=0$, then the definition of $H_V$ 
	would yield 
	\begin{equation*}
	(\varphi, (H_0-\lambda)\psi)=(\varphi, (H_V-\lambda)\psi)
		-\int_{\Real^d}\overline{\varphi}\,V\,\psi=0,
	\end{equation*}
for all  $\varphi\in H^2(\Real^d)$, and consequently, 
$H_0\psi=\lambda\psi$. Unless $\psi=0$, this 
would mean that $\lambda$ is an eigenvalue of $H_0$, which 
is impossible. Hence, we conclude that $\phi\neq0$.

Next, if $\lambda\in\rho(H_0)$, 
then Lemma~\ref{Lem.BS} with $\eps=0$ implies 
\begin{equation}\label{CS.K.la}
	\|\phi\|^2 \, \|K_\lambda\| 
	\geq |(\phi,K_\lambda\phi)| = \|\phi\|^2
\end{equation}
and thus $\|K_\lambda\|\geq1$. 
	
If $\lambda \in \sigma(H_0)$, 
we set $\phi_n := \xi_n \phi$ for every 
positive~$n$, where~$\xi_n(x):=\xi(x/n)$ 
and $\xi \in C_0^\infty(\Real^d)$
is a usual cut-off function satisfying
$\xi(x)=1$ for $|x| \leq 1$ and $\xi(x)=0$ 
for $|x| \geq 2$. As in \eqref{CS.K.la}, 
we have
\begin{equation*}
	\|\phi_n\| \|\phi\| \, 
	\|K_{\lambda+i\eps}\|
	\geq |(\phi_n,K_{\lambda+i\eps}\phi)| 
	\,.
\end{equation*}
In view of $\phi_n\in L^2_0(\Real^d)$, we can invoke 
Lemma~\ref{Lem.BS} and take the limit $\eps \to 0^\pm$ 
to conclude
\begin{equation*}
	\|\phi_n\| \|\phi\| \, 
	\liminf_{\eps\to 0^\pm} \|K_{\lambda+i\eps}\| 
	\geq |(\phi_n,\phi)|
	\,.
\end{equation*}
The desired claim now follows by taking the 
limit $n\to\infty$.
\end{proof}
%
\section{Proofs}\label{Sec.proofs}
\noindent
First, we establish two elementary inequalities
which will be crucial in obtaining sharp 
eigenvalue bounds.

\begin{Lemma}\label{key.ineq.1}
For all non-negative real numbers $p$ and $q$, 
the following inequality holds	
\begin{equation}\label{key.elem.ineq.11}
	\frac{1}{2}
	\bigl(e^{q-p}+e^{p-q}\bigr)\leq e^{p+q}-\sin(p+q).
\end{equation}
\end{Lemma}	
\begin{proof}
Let us rewrite \eqref{key.elem.ineq.11} as 
\begin{equation}\label{key.ineq.for1D}
	e^{-2p}+e^{-2q}+2e^{-(p+q)}\sin(p+q)
	\leq2\,.
\end{equation}
Since \eqref{key.ineq.for1D} is symmetric 
with respect to $p$ and $q$, there is no loss 
of generality in assuming that 
$0\leq p\leq q$. Let us fix $p=p_0\geq0$ and 
analyse the smooth function 
\begin{equation*}
	\Phi(q):=e^{-2p_0}+e^{-2q}
		+2e^{-(p_0+q)}\sin(p_0+q)
\end{equation*}	
on the interval $[p_0,\infty)$. It is easy 
to check that every possible critical 
point $q_0$ of $\Phi$ must satisfy the 
equation
\begin{equation*}
	e^{q_0-p_0}=\cos(p_0+q_0)-\sin(p_0+q_0).
\end{equation*}
In view of this observation, some elementary
calculations yield that
\begin{equation}\label{cr.val}
	\Phi(q_0)
	=2e^{-2p_0}\bigl(1-\sin^2(p_0+q_0)\bigr)
	\leq 2.
\end{equation}
On the other hand, we have 
\begin{equation}\label{end.pt.val}
	\Phi(p_0)=2e^{-2p_0}\bigl(1+\sin(2p_0)\bigr)
	\leq2, \qquad 
	\lim_{q\uparrow+\infty}\Phi(q)=e^{-2p_0}
	\leq 1, 
\end{equation}
the first estimate being the consequence 
of the elementary inequality 
$\sin(t)\leq t \leq e^{t}-1$ for $t\geq0$. 
Hence, we conclude from \eqref{cr.val} and 
\eqref{end.pt.val} that
\begin{equation*}
	\max_{q\geq p_0}\Phi(q)\leq 2,
\end{equation*}
which proves the claim 
in~\eqref{key.ineq.for1D}. 
\end{proof}
\begin{Lemma}\label{key.ineq.2}
For all non-negative real numbers $p$ and $q$, 
the following inequality holds	
\begin{equation}\label{key.elem.ineq.2}
	\frac{1}{2}
	\bigl(e^{q-p}+e^{p-q}\bigr)
	\leq (p^2+q^2)e^{p+q}+\cos(p+q).
\end{equation}
\end{Lemma}	
\begin{proof}
Let us rewrite \eqref{key.elem.ineq.2} as 
\begin{equation}\label{key.ineq.for3D}
	e^{-2p}+e^{-2q}-2e^{-(p+q)}\cos(p+q)
	\leq 2(p^2+q^2)\,.
\end{equation}
Since \eqref{key.ineq.for3D} is symmetric 
with respect to $p$ and $q$, there is no loss 
of generality in assuming that 
$0\leq p\leq q$. Let us fix $p=p_0\geq0$ and 
analyse the smooth function 
\begin{equation*}
	\Phi(q):=e^{-2p_0}+e^{-2q}
	-2e^{-(p_0+q)}\cos(p_0+q)-2(p^2_0+q^2)
\end{equation*}	
on the interval $[p_0,\infty)$. We have
\begin{equation}\label{deriv.Phi}
	\Phi'(q)=-2e^{-2q}+2e^{-(p_0+q)}\cos(p_0+q)
	+2e^{-(p_0+q)}\sin(p_0+q)-4q.
\end{equation}
On the other hand, for all $s, t\in\Real$, we 
have the well-known inequalities
$e^t\geq \cos(t)+\sin(t)$
and $e^s\geq1+s$. Applying these inequalities 
with $t=p_0+q$ and $s=-2q$, we obtain 
\begin{equation}\label{deriv.Phi.est}
	\Phi'(q)\leq -2e^{-2q}+2-4q\leq0.
\end{equation}
Therefore, $\Phi$ is non-increasing on $[0,\infty)$
and thus we have 
\begin{equation*}
	\Phi(q)\leq \Phi(p_0)=2e^{-2p_0}
	+2e^{-2p_0}\cos(2p_0)-4p^2_0\leq0, 
\end{equation*}
for all $q\geq p_0$, where for the last step one 
needs to recall the inequality 
\begin{equation*}
	1-\cos(t)\leq\frac{1}{2}t^2\leq\frac{1}{2}t^2e^t, 
	\quad t\geq0.
\end{equation*}
Hence, for each fixed $p\geq0$, we have $\Phi(q)\leq0$ 
for all $q\geq p$. This proves the desired claim.
\end{proof}
\begin{Remark}\label{Rem.analog.elem.ineq}
We note that the expressions on the 
left-hand-sides	of \eqref{key.ineq.for1D} and 
\eqref{key.ineq.for3D} are non-negative 
for all $p,q\geq0$. In the same way as in 
Lemma~\ref{key.ineq.1} one can show that
\begin{equation}\label{key.elem.ineq.12}
	\frac{1}{2}
	\bigl(e^{q-p}+e^{p-q}\bigr)\leq e^{p+q}+\cos(p+q)
\end{equation}
holds for all $p,q\geq0$.
\end{Remark}

Now we are in a position to establish our theorems.

\subsection{Proof of Theorem~\ref{Thm.EV.bound.unif}}
We start with the case  
$\lambda\in\Com\setminus[0,\infty)$ 
and denote by $k$ the principal square root 
of~$\lambda$ as before. Throughout the proof
we assume the parameter $k$ to be fixed.

First, for each $d\in\{1,2,3\}$, we justify the existence 
of a universal constant $c_d>0$ such that the 
corresponding Green's function of the biharmonic 
operators obeys the following pointwise estimate
\begin{equation}\label{pointwise.estim.Green}
		|\widetilde G_\lambda(x,y)|\leq \frac{c_d}{|k|^{2-d/2}}.	
\end{equation} 

For the case $d=1$, elementary calculations show that 
the inequality
\begin{equation*}
	\big|\sqrt{k} \, e^{-\sqrt{-k}\,|x-y|} - 
	\sqrt{-k} \, e^{-\sqrt{k}\,|x-y|}\big|
	\leq\sqrt{2}\sqrt{|k|}
\end{equation*}
is equivalent to \eqref{key.elem.ineq.11} with 
$p=\Re\sqrt{k} \, |x-y|\geq0$
and $q=\Im\sqrt{k} \, |x-y|\geq 0$ if 
$\arg(\lambda)\in(0,\pi]$
or with 
$p=\Re\sqrt{-k} \, |x-y|\geq0$
and $q=\Im\sqrt{-k} \, |x-y|\geq 0$ if
$\arg(\lambda)\in(-\pi, 0)$.
Hence, we have \eqref{pointwise.estim.Green} with the 
constant $c_1:=\frac{1}{2\sqrt{2}}$, 
\ie
\begin{equation*}
	|\widetilde G_\lambda(x,y)|
	\leq \frac{1}{2\sqrt{2}|k|^{3/2}}.	
\end{equation*}

For the case $d=3$, elementary calculations show
that the inequality
\begin{equation*}
\big|e^{-\sqrt{-k}\,|x-y|} - 
e^{-\sqrt{k}\,|x-y|}\big|
\leq\sqrt{2}\sqrt{|k|}|x-y|
\end{equation*}
is equivalent to \eqref{key.elem.ineq.2} with 
$p=\Re\sqrt{k} \, |x-y|\geq0$
and $q=\Im\sqrt{k} \, |x-y|\geq 0$ if 
$\arg(\lambda)\in(0,\pi]$
or with 
$p=\Re\sqrt{-k} \, |x-y|\geq0$ 
and $q=\Im\sqrt{-k} \, |x-y|\geq 0$ if 
$\arg(\lambda)\in(-\pi, 0)$.
That is why we have \eqref{pointwise.estim.Green} with 
the constant $c_3:=\frac{1}{4\sqrt{2}\pi}$, 
\ie
\begin{equation*}
\begin{aligned}
|\widetilde G_\lambda(x,y)| 
\leq \frac{1}{4\sqrt{2}\pi\sqrt{|k|}}\,.
\end{aligned}
\end{equation*}

For the case $d=2$, on the account of the 
asymptotic expansion of the Macdonald's 
function~$K_0(\zeta)$ for small $\zeta$, 
we can write 
(\cf~\cite[\S~8.447]{Gradshteyn-Ryzhik-2007})
\begin{equation*}
	G_k(x,y)=-\frac{1}{2\pi} \ln\sqrt{-k}-\frac{1}{2\pi}
	\bigg(\gamma+\ln\frac{|x-y|}{2}\bigg)+\text{o}(|x-y|), 
	\qquad |x-y|\to 0^{+},
\end{equation*}
where $\gamma$ is the Euler constant.
Hence, it follows that
\begin{equation*}
\begin{aligned}
	G_k(x,y)-G_{-k}(x,y)&=
		\frac{1}{2\pi}\bigl(\ln\sqrt{k}-\ln\sqrt{-k}\bigr)
		+\text{o}(|x-y|)\\
		&= -\frac{i}{4}+o(|x-y|), 
			\qquad |x-y| \to 0^{+},
\end{aligned}
\end{equation*}
and consequently, recalling \eqref{Green.BiHm}, we get
\begin{equation*}
	\widetilde G_\lambda(x,y)= 
	\frac{i}{8\pi k}\textrm{Arg} (k) +o(|x-y|), 
		\qquad |x-y|\to 0^{+}.
\end{equation*}
On the other hand, in view of the well-known asymptotic expansion 
of the Macdonald's function~$K_0(\zeta)$ for large $\zeta$, we have

\begin{equation*}
	G_k(x,y)=\sqrt{\frac{\pi}{2\sqrt{-k}\,|x-y|}} \, e^{-\sqrt{-k}|x-y|}
		\bigl(1+\mathcal{O}(1)\bigr)
	=\mathcal{O}\Bigl(\frac{1}{\sqrt[4]{|x-y|}}\Bigr),	\qquad |x-y| \to +\infty,
\end{equation*}
and consequently,
\begin{equation*}
	\widetilde G_\lambda(x,y)= \mathcal{O}(1), 
	\qquad |x-y| \to +\infty.
\end{equation*}
Since these asymptotic expansions are uniform in $k$ and 
the Macdonald's function is entire on the right
half-plane, we conclude the existence of a universal
constant $c_2>0$ satisfying \eqref{pointwise.estim.Green}.

In view of \eqref{pointwise.estim.Green}, now we can 
estimate the norm of the Birman-Schwinger 
operator as follows 
\begin{equation}\label{first}
	\|K_\lambda\|^2 
	\leq \|K_\lambda\|^2_\mathrm{HS}
	=\iint_{\Real^d\times\Real^d}|V(x)|\,
		|\widetilde G_\lambda(x,y)|^2
		\,|V(y)|\,dx\,dy 
	\leq \frac{c_d^2}{|k|^{4-d}}\|V\|^2_{L^1(\Real^d)}
\end{equation}
and thus
\begin{equation*}
	\displaystyle
	\|K_\lambda\|
	\leq \frac{c_d\|V\|_{L^1(\Real^d)}}{|\lambda|^{1-d/4}}. 
\end{equation*}
If $\lambda\in(0,\infty)$, then the same analysis 
applied for $\lambda+i\eps$ with $\eps>0$ (instead 
of $\lambda$) yields
\begin{equation*}
	\liminf_{\eps\to 0^\pm} \|K_{\lambda+i\eps}\|
	\leq \liminf_{\eps\to 0^\pm} 
	\frac{c_d\|V\|_{L^1(\Real^d)}}
	{|\lambda+i\eps|^{1-d/4}}=
	\frac{c_d\|V\|_{L^1(\Real^d)}}
	{|\lambda|^{1-d/4}}. 
\end{equation*}
Hence, Corollary~\ref{Corol.BS} implies that 
$\lambda\in\Com\setminus\{0\}$
cannot be an eigenvalue for~$H_V$ unless it holds 
that
\begin{equation*}
	c_d\|V\|_{L^1(\Real^d)}
	\geq|\lambda|^{1-d/4}.
\end{equation*}
This completes the proof since the origin~$\lambda=0$ 
obviously belongs to the right-hand side
of~\eqref{opt.bound.unif}.
\qed

\subsection{Proof of Theorem~\ref{Thm.EV.bound.3D.Rollnik}}
First, assume that 
$\lambda\in\Com\setminus[0,\infty)$ 
and let $k$ be the principal square root of~$\lambda$. 
Elementary calculations show that the estimate
\begin{equation*}
	\big|e^{-\sqrt{-k}\,|x-y|} - 
	e^{-\sqrt{k}\,|x-y|}\big|
	\leq \sqrt{2}
\end{equation*}
is equivalent to \eqref{key.elem.ineq.12} with 
	$p=\Re\sqrt{k} \, |x-y|\geq0$, 
	$q=\Im\sqrt{k} \, |x-y|\geq 0$ if 
	$\arg(\lambda)\in(0,\pi]$, and with 
	$p=\Re\sqrt{-k} \, |x-y|\geq0$, 
	$q=\Im\sqrt{-k} \, |x-y|\geq 0$ if 
	$\arg(\lambda)\in(-\pi, 0)$.
Hence, the Green's 
function can be estimated as 
\begin{equation*}
|\widetilde G_\lambda(x,y)|
\leq \frac{1}{4\sqrt{2}\pi|k||x-y|}	
\end{equation*}
for all distinct $x,y\in\Real^3$. Using this, we can 
estimate the norm of the Birman-Schwinger operator in 
terms of the Rollnik norm of $V$ as follows
\begin{equation*}\label{Est.BS.op.3D}
\begin{aligned}
	\|K_\lambda\|^2 
	\leq \|K_\lambda\|^2_\mathrm{HS}
	=\iint_{\Real^3\times\Real^3}|V(x)|\,
	|\widetilde G_\lambda(x,y)|^2
	\,|V(y)|\,dx\,dy 
	\leq \frac{1}{32\pi^2|k|^2}\|V\|^2_{R}.
\end{aligned}
\end{equation*}
Hence, we have
\begin{equation*}
	\displaystyle
	\|K_\lambda\|\leq 
	\frac{1}{4\pi\sqrt{|\lambda|}}
	\frac{\|V\|_{R}}{\sqrt{2}}.
\end{equation*}
If $\lambda\in(0,\infty)$, then the same analysis 
applied for $\lambda+i\eps$ with $\eps>0$ (instead 
of $\lambda$) yields that
\begin{equation*}
	\liminf_{\eps\to 0^\pm} \|K_{\lambda+i\eps}\|
	\leq \liminf_{\eps\to 0^\pm} 
	\frac{1}{4\pi\sqrt{|\lambda+i\eps|}}
	\frac{\|V\|_{R}}{\sqrt{2}}
	=	\frac{1}{4\pi\sqrt{|\lambda|}}
	\frac{\|V\|_{R}}{\sqrt{2}}\,.
\end{equation*}
Hence, Corollary~\ref{Corol.BS} implies that
$\lambda\in\Com\setminus\{0\}$
cannot be an eigenvalue 
for~$H_V$ if
\begin{equation*}
	\frac{1}{4\pi}\frac{\|V\|_{R}}{\sqrt{2}}
	<\sqrt{|\lambda|}.
\end{equation*}
This completes the proof since $\lambda=0$ 
obviously belongs to the right-hand-side
of \eqref{inclusion.Rollnik}.\qed

\subsection{Proof of Theorem~\ref{Thm.alt.EV.bound.3D}}
First, assume that $\lambda \in \Com \setminus [0,\infty)$ 
and let $k$ be the principal square root of~$\lambda$. 
By the triangle inequality, we have
\begin{equation*} 
	\|K_\lambda\|
	\leq 
	\frac{1}{2|k|} 
	\left[ 
	\||V|^{1/2}(-\Delta-k)^{-1} V_{1/2}\| 
	+ 
	\||V|^{1/2}(-\Delta+k)^{-1} V_{1/2}\| 
	\right]\,.
\end{equation*}
where the integral kernel reads now (\cf~\eqref{Green.Lap})
\begin{equation*}
	(-\Delta-k)^{-1}(x,y) 
	= \frac{e^{-\sqrt{-k}\,|x-y|}}{4\pi \, |x-y|}  
	\,.
\end{equation*}
Using the pointwise estimate
\begin{equation*}
	|(-\Delta-k)^{-1}(x,y)|\leq(-\Delta)^{-1}(x,y)
\end{equation*}
valid for all $x,y \in \Real$ with $ x\not=y$  
and $k\in\Com$, 
we estimate as in~\cite[proof of Lem.~1]{FKV}
\begin{equation*} 
\begin{aligned}
	\| |V|^{1/2}(-\Delta \pm k)^{-1} V_{1/2} \|
	\leq \| |V|^{1/2}(-\Delta)^{-1} |V|^{1/2} \|
	\leq \| |V|^{1/2}(-\Delta)^{-1/2}\| 
	\|(-\Delta)^{-1/2} |V|^{1/2} \| \,.
\end{aligned}
\end{equation*}
Finally, noticing that 
\begin{equation*} 
	\|(-\Delta)^{-1/2} |V|^{1/2} \|
	= \| |V|^{1/2}(-\Delta)^{-1/2}\|
	= \left(
	\sup_{\stackrel[\psi\not=0]{}{\psi \in H^1(\Real^3)}}	 
	\frac{\displaystyle \int_{\Real^3} |V||\psi|^2}
	{\displaystyle \int_{\Real^3} |\nabla\psi|^2} 
	\right)^{1/2}\,,
\end{equation*}
we arrive at the estimate
\begin{equation}\label{norm.est.max}
	\|K_\lambda\|\leq\frac{1}{\sqrt{|\lambda|}}
	\left(
	\sup_{\stackrel[\psi\not=0]{}{\psi \in H^1(\Real^3)}}	 
	\frac{\displaystyle \int_{\Real^3} |V||\psi|^2}
	{\displaystyle \int_{\Real^3} |\nabla\psi|^2} 
	\right)
	\,.
\end{equation}
We notice as before that the origin $\lambda=0$ 
trivially satisfies the estimate in \eqref{bound.3D}.
If $\lambda\in(0,\infty)$, then by repeating the 
analysis for $\lambda+i\eps$ with $\eps>0$ 
(instead of $\lambda$) we see that the right-hand side 
of \eqref{norm.est.max} dominates 
$\liminf_{\eps\to 0^\pm} \|K_{\lambda+i\eps}\|$. 
The claimed inclusion thus immediately follows  
from Corollary~\ref{Corol.BS}.
\qed
%

\section{Optimality of the bounds}\label{Sec.end}
%
In this section we discuss two concrete 
examples which demonstrate the optimality 
of the constants corresponding to the 
eigenvalue bounds for dimensions~$d=1$ and~$d=3$ 
in Theorem~\ref{Thm.EV.bound.unif}. 
Both examples are given in terms of distributional Dirac delta potentials,
but an approximation by regular potentials is also mentioned.
We do not consider the delta potential in two dimensions
because the constant $C_2$ of Theorem~\ref{Thm.EV.bound.unif} 
is not explicit (nonetheless, preliminary computations confirm
that the weak-coupling constant $c_2=\frac{1}{64}$ is achieved by
the Dirac delta potential also for $d=2$).

\subsection{The delta potential in one dimension}
As described in Section~\ref{Sec.pre},
the operator $H_{\alpha\delta} = H_0 \dot{+} \alpha\delta$ in $\sii(\Real)$
with $\alpha \in \Com$ should be understood
as the operator associated with the quadratic form
\begin{equation*}
  h_{\alpha\delta}[\psi] := 
  \int_{\Real} |\psi''|^2 
  +\alpha \, |\psi(0)|^2
  \,, \qquad
  \Dom(h_{\alpha\delta}):=H^{2}(\Real)
  \,.
\end{equation*}
One has 
\begin{equation}\label{Dirac}
\begin{aligned}
  H_{\alpha\delta}\psi(x) &= \psi''''(x)
  \,, \qquad 
  x \in \Real\setminus\{0\}
  \,,
  \\ 
  \Dom(H_{\alpha\delta}) &= 
  \big\{
  \psi \in H^{4}(\Real\setminus\{0\}) 
  \cap H^{3}(\Real)
  \, : \
  \psi'''(0^+) - \psi'''(0^-) = -\alpha \, \psi(0) \big\} 
  \,.
\end{aligned}
\end{equation}
Note that the function values at~$0$ are well defined
due to the Sobolev embedding $H^4(\Real) \hookrightarrow C^3(\Real)$.
Since the perturbation is a point interaction, it follows that 
$\sigma_\mathrm{ess}(H_{\alpha\delta}) = [0,+\infty)$.

Let us look for eigenvalues of~$H_{\alpha\delta}$,
\ie~we consider $H_{\alpha\delta}\psi = k^4 \psi$
with the convention 
\begin{equation}\label{convention}
  \Re k > 0 \qquad \mbox{and} \qquad \Im k > 0
  \,.
\end{equation}
The general solution of the differential equation 
$\psi'''' = k^4\psi$ in any open real interval reads
\begin{equation}\label{ansatz}
  \psi(x) = C_1 e^{k x} + C_2 e^{-k x}
  + C_3 e^{i k x} + C_4 e^{- i k x}
  \,,
\end{equation}
where $C_1,C_2,C_3,C_4$ are arbitrary complex numbers.
Considering these general solution in the separate intervals
$(-\infty,0)$ and $(0,+\infty)$, and imposing the integrability requirements 
($C_1,C_4=0$ if $x>0$ and $C_2,C_3=0$ if $x<0$)
together with the interface conditions at zero due to~\eqref{Dirac}
(continuity up to the second derivative and the jump condition
for the third derivative), we arrive at the condition 
\begin{equation*}\label{condition1}
  k^3 = \frac{1}{2^{3/2}} \, \alpha \, e^{-i\pi/4}
  \,.
\end{equation*}
In order to make this equation compatible with~\eqref{convention},
we see that the point spectrum of~$H_{\alpha\delta}$ is not empty
if, and only if, $\arg\alpha \in (\pi/4,7\pi/4)$,
or equivalently, $\Re\alpha < |\Im\alpha|$.
If this condition is satisfied, $H_{\alpha\delta}$~possesses
one eigenvalue corresponding to the right-hand side
put to the power~$4/3$. 
Summing up,
\begin{equation*} 
  \sigma_\mathrm{p}(H_{\alpha\delta}) = 
  \begin{cases}
  \displaystyle
  \left\{
  \frac{1}{4} \, \alpha^{4/3} \, e^{-i\pi/3}
  \right\}
  & \mbox{if} \quad \Re\alpha < |\Im\alpha| \,,
  \\
  \varnothing
  & \mbox{if} \quad \Re\alpha \geq |\Im\alpha| \,.
  \end{cases}
\end{equation*}
Notice that the arguments of~$\alpha$ for which the point spectrum is empty
correspond to real points intersecting the half-axis $[0,+\infty)$.
So any eigenvalue of~$H_{\alpha\delta}$, if it exists,
is necessarily discrete.

In particular, choosing $\alpha := e^{i\theta}$ 
and varying $\theta \in (\pi/4,7\pi/4)$,
any boundary point of the closed disk 
$\{\lambda \in \Com : |\lambda|\leq \frac{1}{4}\}$
except for the point $\lambda = \frac{1}{4}$
will be an eigenvalue of~$H_{\alpha\delta}$.
This proves the optimality of 
Theorem~\ref{Thm.EV.bound.unif} for $d=1$ because 
$\|\alpha\delta\|_{L^1(\Real)} = |\alpha| = 1$.

\begin{Remark}
One can also justify the optimality of the constant 
in the eigenvalue bound corresponding to the 
inclusion \eqref{opt.bound.unif} for $d=1$ in the 
limit $\eps \to 0^+$ by considering 
the following family of regular but 
singularly scaled potentials
\begin{equation*}
	\delta_\eps(x) :=
	\begin{cases}
	\displaystyle
	\frac{1}{\eps} 
	& \displaystyle \mbox{if} 
	\quad |x| < \eps/2 \,,
	\\
	0 &\mbox{otherwise} \,. 
	\end{cases} 
\end{equation*}
Of course $\delta_\eps \to \delta$ 
in the sense of distributions as $\eps \to 0^+$.
What is more, the point spectrum of $H_{\alpha\delta_\eps}$
converges to the point spectrum of $H_{\alpha\delta}$ as $\eps \to 0^+$.
This can be checked by explicitly solving the differential equations 
of the eigenvalue problem
in the separate intervals where~$\delta_\eps$ is constant
and matching the solutions at the boundary points. 
\end{Remark}

\subsection{The delta potential in three dimensions}
In analogy with the one-dimensional situation above,
we consider the operator 
$H_{4\pi\alpha\delta} = H_0 \dot{+} 4\pi\alpha\delta$ in $\sii(\Real^3)$
with $\alpha \in \Com$ associated with the quadratic form
\begin{equation*}
	h_{4\pi\alpha\delta}[\psi] := 
	\int_{\Real^3} |\Delta\psi|^2 
	+ 4\pi \alpha\, |\psi(0)|^2
	\,, \qquad
	\Dom(h_{4\pi\alpha\delta}):=H^2(\Real^3)
	\,.
\end{equation*}
Again, $\sigma_\mathrm{ess}(H_{4\pi\alpha\delta}) = [0,+\infty)$.

Let us look for eigenvalues of $H_{4\pi\alpha\delta}$ 
corresponding to radially symmetric eigenfunctions. 
That is, we consider the equation $H_{4\pi\alpha\delta}\psi = k^4 \psi$
with the same convention~\eqref{convention} as above
and look for special solutions of the form $\psi(x) = g(|x|)$ 
with $g \in H^2((0,\infty),r^2 \, dr)$.
These eigenvalues are determined by eigenvalues 
of the one dimensional operator $T_\alpha$ in $L^2((0,\infty),r^2 \, dr)$
associated with the quadratic form 
\begin{equation*}
	t_\alpha[g] := 
	\int_0^\infty |r^{-1}[r g'(r)]'|^2 \, r^2 \, dr
	+ \alpha \, |g(0)|^2
	\,, \qquad
	\Dom(t_\alpha):=H^2((0,\infty),r^2\,dr)  
        \,.
\end{equation*}
Using the unitary transform 
$U : L^2((0,\infty),r^2 \, dr) \to L^2((0,\infty),dr)$
which acts as $(Ug)(r) := r g(r)$,
the operator~$T_\alpha$ is unitarily equivalent
to the operator $\widetilde{T}_\alpha := U T_\alpha U^{-1}$ 
in $L^2((0,\infty))$ associated with the quadratic form
\begin{equation*}
	\widetilde{t}_\alpha[f] := 
	\int_0^\infty |f''(r)|^2 \, dr
	+ \alpha \, |f'(0)|^2
	\,, \qquad
	\Dom(\widetilde{t}_\alpha) := H^2((0,\infty)) \cap H_0^1((0,\infty)) 
        \,.
\end{equation*}
One has
\begin{equation*}\label{Dirac.prime}
\begin{aligned}
	\widetilde{T}_\alpha f(r) &= f''''(r)
	\,, \qquad 
	x \in (0,\infty)
	\,,
	\\ 
	\Dom(\widetilde{T}_\alpha) &= 
	\big\{
	\eta \in H^4((0,\infty)) 
	\cap H^1_0((0,\infty))
	\ : \
	f''(0) - \alpha f'(0) = 0
\big\} \,.
\end{aligned}
\end{equation*}
The boundary values at~$0$ are well defined due to
the Sobolev embedding $H^4((0,\infty)) \hookrightarrow C^3([0,\infty))$.
Considering the general solution~\eqref{ansatz} 
of the differential equation $f''''=k^4 f$  
and imposing the integrability requirement (\ie\ $C_1,C_4=0$)
together with the boundary conditions at~$0$,
we arrive at the condition 
\begin{equation*}\label{condition3}
  k = \frac{1}{2^{1/2}} \, \alpha \, e^{i\pi/4}
  \,.
\end{equation*}
In order to make this equation compatible with~\eqref{convention},
we see that the point spectrum of~$T_{\alpha}$ is not empty
if, and only if, $\arg\alpha \in (3\pi/4,5\pi/4)$,
or equivalently, $\Re\alpha < -|\Im\alpha|$.
If this condition is satisfied, $T_\alpha$~possesses
one eigenvalue corresponding to the right-hand side
put to the power~$4$. 
Summing up,
\begin{equation*} 
  \sigma_\mathrm{p}(T_{\alpha}) = 
  \begin{cases}
  \displaystyle
  \left\{
  -\frac{1}{4} \, \alpha^{4} 
  \right\}
  & \mbox{if} \quad \Re\alpha < -|\Im\alpha| \,,
  \\
  \varnothing
  & \mbox{if} \quad \Re\alpha \geq -|\Im\alpha| \,.
  \end{cases}
\end{equation*}
Notice that the arguments of~$\alpha$ for which the point spectrum is empty
correspond to real points intersecting the half-axis 
$[0,+\infty) = \sigma_\mathrm{ess}(T_\alpha)$.
So any eigenvalue of~$T_{\alpha}$, if it exists,
is necessarily discrete.

In particular, choosing $\alpha := e^{i\theta}$ 
and varying $\theta \in (3\pi/4,5\pi/4)$,
any boundary point of the closed disk 
$\{\lambda \in \Com : |\lambda|\leq \frac{1}{4}\}$
except for the point $\lambda = \frac{1}{4}$
will be an eigenvalue of~$H_{4\pi\alpha\delta}$.
This proves the optimality of 
Theorem~\ref{Thm.EV.bound.unif} for $d=3$
because $\|4\pi\alpha\delta\|_{L^1(\Real^3)} = 4\pi|\alpha| = 4\pi$.

\begin{Remark}
One can also justify the optimality of the constant 
in the eigenvalue bound corresponding to the 
inclusion~\eqref{opt.bound.unif} for $d=3$ in the 
limit $\eps \to 0^+$ by considering 
the following family of regular but 
singularly scaled potentials
\begin{equation*}
	\delta_\eps(x) :=
	\begin{cases}
	\displaystyle
	\frac{1}{4\pi\eps^3} 
	& \displaystyle \mbox{if} 
	\quad |x| < \eps \,,
	\\
	0 &\mbox{otherwise} \,. 
	\end{cases} 
\end{equation*}
Of course $\delta_\eps \to \delta$ 
in the sense of distributions as $\eps \to 0^+$.
What is more, restricting to radially symmetric eigenfunctions, 
the point spectrum of $H_{4\pi\alpha\delta_\eps}$
converges to the point spectrum of $H_{4\pi\alpha\delta}$ as $\eps \to 0^+$.
This can be checked by explicitly solving 
the radial parts of the differential equations of the eigenvalue problem
in the separate intervals where~$\delta_\eps$ is constant
and matching the solutions at the interface point $r=\eps$. 
\end{Remark}
\subsection*{Acknowledgment}
The research of D.K. was partially supported by the GACR grant No.~18-08835S.

%
%
\bibliography{bib_BiHM}
\bibliographystyle{amsplain}
\end{document}